\documentclass[11pt]{amsart}

\usepackage[utf8]{inputenc}
\usepackage{microtype}
\usepackage{amsfonts}
\usepackage{amsmath}
\usepackage{graphicx}
\usepackage{float}
\usepackage[dvipsnames]{xcolor}
\usepackage{hyperref}

\usepackage{amssymb}
\usepackage{enumitem}
\usepackage{mathrsfs}

\usepackage{a4wide}

\usepackage{empheq}

\newtheorem{theorem}{Theorem}[section]
\newtheorem*{theorem*}{Theorem}
\newtheorem{lemma}[theorem]{Lemma}
\newtheorem{proposition}[theorem]{Proposition}
\newtheorem*{proposition*}{Proposition}
\newtheorem{corollary}[theorem]{Corollary}
\newtheorem*{corollary*}{Corollary}
\newtheorem{conjecture}[theorem]{Conjecture}
\newtheorem{cit}[theorem]{Citation}
\newtheorem*{conjecture*}{Conjecture}
\newtheorem{question}[theorem]{Question}
\newtheorem*{question*}{Question}

\theoremstyle{definition}
\newtheorem{definition}[theorem]{Definition}
\newtheorem*{definition*}{Definition}
\newtheorem{remark}[theorem]{Remark}

\newcommand{\N}{\mathbb{N}}
\newcommand{\Z}{\mathbb{Z}}

\newcommand{\R}{\mathbb{R}}
\newcommand{\Cantor}{\mathfrak{C}}

\DeclareMathOperator{\GL}{GL}
\DeclareMathOperator{\Sym}{Sym}
\DeclareMathOperator{\Symfin}{Sym_{fin}}
\DeclareMathOperator{\F}{F}
\DeclareMathOperator{\Stab}{Stab}
\DeclareMathOperator{\id}{id}

\numberwithin{equation}{section}

\begin{document}

\title{Finite presentability of twisted Brin--Thompson groups}
\date{\today}
 %simple groups, top methods in gp thry


\author[M.~C.~B.~Zaremsky]{Matthew C.~B.~Zaremsky}
\address{Department of Mathematics and Statistics, University at Albany (SUNY), Albany, NY}
\email{mzaremsky@albany.edu}

\begin{abstract}
Given a group $G$ acting faithfully on a set $S$, we characterize precisely when the twisted Brin--Thompson group $SV_G$ is finitely presented. The answer is that $SV_G$ is finitely presented if and only if we have the following: $G$ is finitely presented, the action of $G$ on $S$ has finitely many orbits of two-element subsets of $S$, and the stabilizer in $G$ of any element of $S$ is finitely generated. Since twisted Brin--Thompson groups are simple, a consequence is that any subgroup of a group admitting an action as above satisfies the Boone--Higman conjecture. In the course of proving this, we also establish a sufficient condition for a group acting cocompactly on a simply connected complex to be finitely presented, even if certain edge stabilizers are not finitely generated, which may be of independent interest.
\end{abstract}

\maketitle
\thispagestyle{empty}

\section*{Introduction}

Twisted Brin--Thompson groups are a family of groups $SV_G$, parameterized by an arbitrary group $G$ and a set $S$ on which $G$ acts faithfully. These groups were introduced by Belk and the author in \cite{belk22}, as a way of ``twisting'' the Brin--Thompson groups introduced by Brin in \cite{brin04}. The original Brin--Thompson groups amount to the $G=\{1\}$ case. Perhaps the most important fact about twisted Brin--Thompson groups is that they are always simple, for any $G$ and $S$ \cite[Theorem~3.4]{belk22}. It is also easy to characterize when $SV_G$ is finitely generated: this happens if and only if $G$ is finitely generated and the action of $G$ on $S$ has finitely many orbits \cite[Theorem~A]{belk22}. In particular, twisted Brin--Thompson groups provide an easy way to embed an arbitrary finitely generated group $G$ into a finitely generated simple group $SV_G$, for appropriate choice of $S$.

It is much more difficult to characterize when $SV_G$ is finitely presented. In \cite[Theorem~D]{belk22} a sufficient condition was given, with restrictions on orbits and stabilizers for all finite subsets of $S$. We also posed a conjecture, the $n=2$ case of \cite[Conjecture~H]{belk22}, predicting that these conditions could be relaxed to only involve orbits of subsets of size $2$ and stabilizers of subsets of size $1$. In this paper we prove this conjecture, thus characterizing when a twisted Brin--Thompson group is finitely presented. Let us state this more precisely. Let $G$ be a group acting on a set $S$. We will say that the action is of \emph{type~(A)} if the following holds:

\begin{quote}
\textbf{(A):} The action is faithful, the group $G$ is finitely presented, each $\Stab_G(s)$ for $s\in S$ is finitely generated, and there are finitely many $G$-orbits of two-element subsets of $S$.
\end{quote}

Our main theorem is now the following:

{\renewcommand*{\thetheorem}{\Alph{theorem}}
\begin{theorem}\label{thrm:main}
Let $G$ be a group acting faithfully on a set $S$. Then the twisted Brin--Thompson group $SV_G$ is finitely presented if and only if the action of $G$ on $S$ is of type~(A).
\end{theorem}}

As a remark, the sufficient conditions in \cite{belk22} to get $SV_G$ to be finitely presented, assuming $G$ is, required that $\Stab_G(T)$ be finitely presented for all finite $T\subseteq S$, and that for each $k\in\N$ there are finitely many $G$-orbits of $k$-element subsets. In \cite{zar_taste} finite presentability of the stabilizers was relaxed to finite generation. The ``if'' part of our main theorem here is therefore a significant improvement on these existing results, since now one only needs to understand stabilizers of single points, and orbits of pairs.

We should also point out one crucial aspect of our proof of Theorem~\ref{thrm:main}, which is a generalization of the so called Stein complex $X$ of $SV_G$ from \cite{belk22}, to a family of subcomplexes $X(k)$ ($k\in\N)$ that are significantly smaller. These may no longer be contractible like $X$, but are still simply connected for large enough $k$, and thus still useful for trying to deduce finite presentability of $SV_G$.

\medskip

An immediate consequence of Theorem~\ref{thrm:main} is the following sufficient condition for a group to satisfy the Boone--Higman conjecture.

{\renewcommand*{\thetheorem}{\Alph{theorem}}
\begin{corollary}\label{cor:main}
Any subgroup of a group admitting an action of type~(A) (has solvable word problem and) satisfies the Boone--Higman conjecture. \qed
\end{corollary}}

The \emph{Boone--Higman conjecture} predicts that a finitely generated group has solvable word problem if and only if it embeds into a finitely presented simple group. The backward direction is easy, but the forward direction has been open for decades. This conjecture was posed by Boone and Higman in the early 1970's \cite{boone74}, and has attracted renewed attention in recent years. Prominent examples of groups known to satisfy the conjecture include $\GL_n(\Z)$ \cite{scott84}, hence virtually special and virtually nilpotent groups, and also hyperbolic groups \cite{bbmz_hyp}, finitely presented or contracting self-similar groups \cite{bbmz_hyp,zaremsky_fpss,belkmatuccicontracting}, and Baumslag--Solitar groups and (finite rank free)-by-cyclic groups \cite{bux_boonehigman}. See \cite{bbmz_survey} for more on the history and progress around this conjecture.

From now on, whenever we say that a group, ``satisfies the Boone--Higman conjecture,'' we are implicity saying that it has solvable word problem, and explicitly saying that it embeds into a finitely presented simple group. We should mention that we are not aware of any examples of groups for which the Boone--Higman conjecture is known to hold but the group does not embed into a group admitting an action of type~(A). Thus, it is an interesting question (see Question~\ref{quest:BHC_PBHC}) whether embedding into a group admitting an action of type~(A) could possibly be equivalent to embedding into a finitely presented simple group. In Section~\ref{sec:bh} we give some new examples of groups admitting actions of type~(A), and discuss some general results related to the Boone--Higman conjecture.

We should mention that Theorem~\ref{thrm:main} has a direct connection to finite presentability of permutational wreath products. By \cite{cornulier06}, the permutational wreath product $W\wr_S G \coloneqq \bigoplus_S W \rtimes G$ is finitely presented if and only if the action of $G$ on $S$ is of type~(A) and $W\ne\{1\}$ is finitely presented. Thus, a more concise way to phrase Theorem~\ref{thrm:main} is to say that $SV_G$ is finitely presented if and only if $W\wr_S G$ is finitely presented, for any choice of finitely presented non-trivial $W$, for example $W=\Z$. (Note that in \cite{cornulier06} the action of $G$ on $S$ need not be faithful, but with twisted Brin--Thompson groups we always want the action to be faithful.) As a remark, in \cite{bbmz_hyp} it is explained that $SV_G$ is the ``topological full group'' of the wreath product $V\wr_S G$, so the above is especially relevant for $W=V$.

Finally, let us emphasize one new technical result that we use to prove the relevant $SV_G$ are finitely presented, namely Proposition~\ref{prop:technical}. It is well known that if a group acts cocompactly on a simply connected complex, with finitely presented vertex stabilizers and finitely generated edge stabilizers, then the group is finitely presented. In Proposition~\ref{prop:technical}, we find a way to still achieve finite presentability of the group even if certain edge stabilizers are not finitely generated, which could be of independent interest.

\medskip

This paper is organized as follows. In Section~\ref{sec:tbt} we recall some background on twisted Brin--Thompson groups $SV_G$. In Section~\ref{sec:qr} we prove that $\Z\wr_S G$ is a quasi-retract of $SV_G$, thus establishing the ``only if'' direction of Theorem~\ref{thrm:main}. In Section~\ref{sec:fp} we prove a criterion, Proposition~\ref{prop:technical}, for deducing finite presentability of a group from its action on a complex, even if some edge stabilizers are not finitely generated. This leads in Section~\ref{sec:stein} to the proof of the ``if'' direction of Theorem~\ref{thrm:main}, using a variation of the Stein complex for $SV_G$. Finally, in Section~\ref{sec:bh} we discuss some implications for the Boone--Higman conjecture.

\subsection*{Acknowledgments} Many thanks are due to Jim Belk and Francesco Fournier-Facio; to Jim in particular for the idea of the quasi-retraction $\rho_{\kappa_0}$ in Section~\ref{sec:qr}, and to Francesco in particular for the (idea for and) proof of Proposition~\ref{prop:commensurable}. I am also grateful to the referee for many helpful comments and suggestions. The author is supported by grant \#635763 from the Simons Foundation.

%-------------------------------------------------
\section{Twisted Brin--Thompson groups and groupoids}\label{sec:tbt}

Let $\Cantor=\{0,1\}^\N$ be the usual Cantor set. For a set $S$, we can consider the \emph{Cantor cube} $\Cantor^S$, that is the set of all functions from $S$ to $\Cantor$, with the usual product topology. Let $\{0,1\}^*$ be the set of all finite binary sequences, and write $\varnothing$ for the empty word. Given a function $\psi\colon S\to\{0,1\}^*$ such that $\psi(s)=\varnothing$ for all but finitely many $s\in S$, we define the \emph{dyadic brick} $B(\psi)$ to be
\[
B(\psi)\coloneqq \{\kappa \in \Cantor^S\mid \psi(s) \text{ is a prefix of } \kappa(s) \text{ for each }s\in S\}\text{.}
\]
This is a (basic) open set in $\Cantor^S$. There is a \emph{canonical homeomorphism} $h_\psi \colon \Cantor^S \to B(\psi)$ given by
\[
h_\psi(\kappa)(s) \coloneqq \psi(s) \cdot \kappa(s) \text{.}
\]
Any composition $h_\psi \circ h_\varphi^{-1}$ will also be called a \emph{canonical homeomorphism}, from one dyadic brick to another.

\begin{definition}[Brin--Thompson group]
The \emph{Brin--Thompson group} $SV$ is the group of homeomorphisms of $\Cantor^S$ obtained by partitioning $\Cantor^S$ into dyadic bricks in two ways, $B(\varphi_1),\dots,B(\varphi_n)$ and $B(\psi_1),\dots,B(\psi_n)$, and mapping each $B(\varphi_i)$ to $B(\psi_i)$ via the canonical homeomorphism.
\end{definition}

When $S$ is finite, say $|S|=n$, we will write $nV$. These groups were first introuced by Brin in \cite{brin04}, as an infinite family generalizing Thompson's group $V=1V$.

Given a group $G$ acting faithfully on $S$, for each $\gamma\in G$ define the \emph{twist homeomorphism} $\tau_\gamma$ of $\Cantor^S$ to be the homeomorphism permuting the coordinates via $\gamma$, i.e.,
\[
\tau_\gamma(\kappa)(s) \coloneqq \kappa(\gamma^{-1}s) \text{.}
\]
More elegantly, this means $\tau_\gamma(\kappa)(\gamma s)=\kappa(s)$. As in \cite{belk22}, we view the action of $G$ on $S$ as a left action, and the left/right, inverse/not inverse conventions here ensure that $\tau_{\gamma\gamma'}=\tau_\gamma \tau_{\gamma'}$. We will also say \emph{twist homeomorphism} for anything of the form $h_\psi \circ \tau_\gamma \circ h_\varphi^{-1}$, so this is the twist homeomorphism from $B(\varphi)$ to $B(\psi)$ using $\gamma$.

Now we can define twisted Brin--Thompson groups, first introduced in \cite{belk22}. See also \cite{zar_taste} for a shorter introduction.

\begin{definition}[Twisted Brin--Thompson group]
The \emph{twisted Brin--Thompson group} $SV_G$ is the group of homeomorphisms of $\Cantor^S$ obtained by partitioning $\Cantor^S$ into dyadic bricks in two ways, $B(\varphi_1),\dots,B(\varphi_n)$ and $B(\psi_1),\dots,B(\psi_n)$, and mapping each $B(\varphi_i)$ to $B(\psi_i)$ via a twist homeomorphism using some $\gamma_i\in G$.
\end{definition}

Now let $\Cantor^S(m)$ be the disjoint union of $m$ copies of the Cantor cube $\Cantor^S$. We have a notion of dyadic bricks in $\Cantor^S(m)$, just by taking dyadic bricks in the cubes, and thus we have canonical homeomorphisms and twist homeomorphisms even between dyadic bricks in different cubes.

\begin{definition}[(Twisted) Brin--Thompson groupoids, rank, corank]
The \emph{Brin--Thompson groupoid} $S\mathcal{V}$ is the groupoid of all homeomorphisms from some $\Cantor^S(m)$ to some $\Cantor^S(n)$ given by partitioning the domain into some number of dyadic bricks, partitioning the codomain into the same number of dyadic bricks, and sending the domain bricks to the codomain bricks via canonical homeomorphisms. If we instead use twist homeomorphisms, we get the \emph{twisted Brin--Thompson groupoid} $S\mathcal{V}_G$. A homeomorphism from $\Cantor^S(m)$ to $\Cantor^S(n)$ has \emph{rank} $n$ and \emph{corank} $m$.
\end{definition}

Note that $SV_G$ is the subgroup of $S\mathcal{V}_G$ consisting of all elements with rank and corank $1$. Let us also record here the definition of some other elements that will turn out to be important later, namely, simple splits.

\begin{definition}[Simple split]\label{def:simple_split}
The \emph{simple split} $x_s$ for $s\in S$ is the element of $S\mathcal{V}$ with corank $1$ and rank $2$ given by partitioning $\Cantor^S$ into the dyadic bricks $B(\psi_0)$ and $B(\psi_1)$, where $\psi_i\colon S\to \{0,1\}^*$ sends $s$ to $i$ and all other $s'$ to $\varnothing$, and then mapping $B(\psi_0)$ to the first cube of $\Cantor^S(2)$ and $B(\psi_1)$ to the second cube, both via canonical homeomorphisms. 
\end{definition}

%-------------------------------------------------
\section{Quasi-retractions}\label{sec:qr}

In this section we prove the forward direction of Theorem~\ref{thrm:main}:

\begin{proposition}\label{prop:forward}
If $SV_G$ is finitely presented, then the action of $G$ on $S$ is of type~(A).
\end{proposition}

As a first step, we relate condition~(A) to the permutational wreath product $\Z\wr_S G$. Given a group $G$ acting on a set $S$, and a group $W$, the \emph{permutational wreath product} $W\wr_S G$ is the semidirect product
\[
W\wr_S G \coloneqq \bigoplus\limits_{s\in S}W \rtimes G \text{,}
\]
where the action of $G$ on the direct sum is given by permuting coordinates; more precisely, $\gamma\in G$ sends $(w_s)_{s\in S}$ to $(w_{\gamma^{-1} s})_{s\in S}$. Viewing elements of $\bigoplus_S W$ as functions from $S$ to $W$, this action amounts to precomposition by $\gamma^{-1}$, so it is a left action. Cornulier proved a precise characterization of when $W\wr_S G$ is finitely presented, namely:

\begin{cit}\cite[Theorem~1.1]{cornulier06}\label{cit:cornulier}
For $W\ne\{1\}$, the group $W\wr_S G$ is finitely presented if and only if $W$ is finitely presented and the action of $G$ on $S$ is of type~(A).
\end{cit}

Note that in \cite{cornulier06} the condition of having finitely many orbits of two-element subsets of $S$ is phrased as having finitely many orbits in $S\times S$ under the diagonal action of $G$, but this is equivalent.

Since $\Z$ is non-trivial and finitely presented, Citation~\ref{cit:cornulier} implies that if $\Z\wr_S G$ is finitely presented, then the action of $G$ on $S$ is of type~(A). Thus, to prove Proposition~\ref{prop:forward} it suffices to prove that if $SV_G$ is finitely presented then so is $\Z\wr_S G$. We will do this by showing that $\Z\wr_S G$ is a quasi-retract of $SV_G$.

\begin{definition}[Quasi-retract]
A \emph{quasi-retraction} $\rho\colon X\to Y$ from a metric space $X$ to a metric space $Y$ is a coarse Lipschitz function such that there exists a coarse Lipschitz function $\zeta\colon Y\to X$ with $\rho\circ\zeta$ uniformly close to the identity on $Y$. In this case we call $Y$ a \emph{quasi-retract} of $X$.
\end{definition}

Here a function $f\colon X\to Y$ is \emph{coarse Lipschitz} if there exist constants $C,D>0$ such that $d(f(x),f(x'))\le C d(x,x') + D$ for all $x,x'\in X$.

If $G$ and $H$ are finitely generated groups, viewed as metric spaces with word metrics coming from some finite generating sets, then Alonso proved in \cite{alonso94} that if $G$ is of type $\F_n$ and $H$ is a quasi-retract of $G$ then $H$ is of type $\F_n$. In particular when $n=2$ we get that every quasi-retract of a finitely presented group is finitely presented.

\medskip

Let us now begin to construct a quasi-retraction $SV_G \to \Z\wr_S G$. The idea of this quasi-retraction is due to Jim Belk. Let $h\in SV_G$, so $h$ is a homeomorphism from $\Cantor^S$ to itself given by partitioning the domain into dyadic bricks $B(\varphi_1),\dots,B(\varphi_n)$, partitioning the range into dyadic bricks $B(\psi_1),\dots,B(\psi_n)$, and mapping each $B(\varphi_i)$ to $B(\psi_i)$ via a twist homeomorphism using some $\gamma_i\in G$. Let $\kappa\in \Cantor^S$, say $\kappa\in B(\varphi_i)$, so $h(\kappa)\in B(\psi_i)$. For each $s\in S$, let $d_\kappa^s(h)$ be the length of $\varphi_i(s)\in \{0,1\}^*$ and let $r_{h(\kappa)}^s(h)$ be the length of $\psi_i(s)\in \{0,1\}^*$. (Note that $d_\kappa^s(h)$ and $r_{h(\kappa)}^s(h)$ depend on the choice of domain and range partitions, and are not actually well defined solely in terms of $h$; including $\varphi_i$ and $\psi_i$ in the notation is just too bulky.) Note that $d_\kappa^s(h)$ and $r_{h(\kappa)}^s(h)$ are zero for all but finitely many $s\in S$, and intuitively they are measuring how many times the dyadic bricks containing $\kappa$ and its image have been halved in the $s$ dimension, which can be viewed as a sort of ``depth'' in that dimension.

Finally, with all the above notation, define
\[
\rho_\kappa \colon SV_G \to \Z\wr_S G \quad \text{ via } \quad h \mapsto ((r_{h(\kappa)}^{s}(h)-d_\kappa^{\gamma_i^{-1}s}(h))_{s\in S},\gamma_i) \text{.}
\]
Intuitively, this measures the ``change in depth'' at $\kappa$, in every dimension, and also records the twist used at $\kappa$.

To see why these $\rho_\kappa$ are useful for getting a quasi-retraction from $SV_G$ to $\Z\wr_S G$, we need a series of lemmas.

\begin{lemma}
The function $\rho_\kappa$ is well defined.
\end{lemma}

\begin{proof}
We need to prove that the measurements $r_{h(\kappa)}^{s}(h)-d_\kappa^{\gamma_i^{-1}s}(h)$ for each $s\in S$, and the element $\gamma_i$, only depend on $h$, and not on the choice of partitions into dyadic bricks. It suffices to prove that refining the partitions does not change these measurements, and for this it is enough to consider a domain partition refinement that just partitions $B(\varphi_i)$ into two halves in some dimension; see \cite[Remark~1.1]{belk22}. More precisely, fix some $s_0\in S$, and let $\varphi_i^0,\varphi_i^1\colon S\to\{0,1\}^*$ be the functions satisfying $\varphi_i^0(s)=\varphi_i^1(s)=\varphi_i(s)$ for all $s\ne \gamma_i^{-1} s_0$, and $\varphi_i^0(\gamma_i^{-1} s_0)=\varphi_i(\gamma_i^{-1} s_0)0$ and $\varphi_i^1(\gamma_i^{-1} s_0)=\varphi_i(\gamma_i^{-1} s_0)1$. Since $\kappa\in B(\varphi_i)$, either $\kappa\in B(\varphi_i^0)$ or $\kappa\in B(\varphi_i^1)$. Viewing $h$ using this new domain partition, we have that $d_\kappa^{\gamma_i^{-1} s}(h)$ does not change unless $s=\gamma_i^{-1} s_0$, in which case it goes up by one. Now consider the resulting new range partition. The image of $B(\varphi_i^0)$ under $h$ is $B(\psi_i^0)$ or $B(\psi_i^1)$, where $\psi_i^0,\psi_i^1\colon S\to \{0,1\}^*$ are defined by $\psi_i^0(s)=\psi_i^1(s)=\psi_i(s)$ for all $s\ne s_0$, and $\psi_i^0(s_0)=\psi_i(s_0)0$ and $\psi_i^1(s_0)=\psi_i(s_0)1$ (intuitively, twisting via $\gamma_i$ takes the ``cut'' in the $\gamma_i^{-1} s$ dimension to a cut in the $s$ dimension). Now $r_{h(\kappa)}^s(h)$ is the same as before, unless $s=s_0$ in which case it goes up by one. In particular, for any $s\in S$, the difference $r_{h(\kappa)}^{s}(h)-d_\kappa^{\gamma_i^{-1} s}(h)$ is the same whether we use the original partitions or the refinements. Finally, it is obvious that the element $\gamma_i$ does not depend on the choice of partitions (to use the language of \cite[Section~2]{belk22}, $\gamma_i$ is the ``germinal twist'' of $h$ at $\kappa$), so we conclude that $\rho_\kappa$ is well defined.
\end{proof}

\begin{lemma}\label{lem:rho_prod}
For any $h,g\in SV_G$ and any $\kappa\in \Cantor^S$ we have $\rho_\kappa(hg)=\rho_{g(\kappa)}(h)\rho_\kappa(g)$.
\end{lemma}

\begin{proof}
Choose the range partition for $g$ to equal the domain partition for $h$; say in the domain of $g$ we have $B(\varphi_j)$, in the range of $g$ and domain of $h$ we have $B(\psi_j)$, and in the range of $h$ we have $B(\chi_j)$. Say $B(\varphi_i)$ contains $\kappa$, and let $\gamma_i\in G$ be such that $g$ sends $B(\varphi_i)$ to $B(\psi_i)$ via the twist homeomorphism using $\gamma_i$. Similarly choose $\delta_i\in G$ such that $h$ takes $B(\psi_i)$ to $B(\chi_i)$ via the twist homeomorphism using $\delta_i$. Note that $hg$ takes $B(\varphi_i)$ to $B(\chi_i)$ via the twist homeomorphism using $\delta_i \gamma_i$. Now we compute
\begin{align*}
\rho_\kappa(hg) &= \bigg(\big(r_{hg(\kappa)}^{s}(hg) - d_\kappa^{(\delta_i \gamma_i)^{-1} s}(hg)\big)_{s\in S},\delta_i \gamma_i\bigg)\\
&= \bigg(\big(r_{hg(\kappa)}^{s}(h)-d_\kappa^{\gamma_i^{-1}\delta_i^{-1}s}(g)\big)_{s\in S},\delta_i \gamma_i\bigg)\\
&= \bigg(\big(r_{hg(\kappa)}^{s}(h)-d_{g(\kappa)}^{\delta_i^{-1} s}(h) + r_{g(\kappa)}^{\delta_i^{-1} s}(g)-d_\kappa^{\gamma_i^{-1}\delta_i^{-1}s}(g)\big)_{s\in S}, \delta_i \gamma_i\bigg)\\
&= \bigg(\big(r_{hg(\kappa)}^{s}(h)-d_{g(\kappa)}^{\delta_i^{-1}s}(h)\big)_{s\in S},1\bigg) \bigg(\big(r_{g(\kappa)}^{\delta_i^{-1}s}(g)-d_\kappa^{\gamma_i^{-1}\delta_i^{-1}s}(g)\big)_{s\in S},\delta_i\gamma_i\bigg)\\
&= \bigg(\big(r_{hg(\kappa)}^{s}(h)-d_{g(\kappa)}^{\delta_i^{-1}s}(h)\big)_{s\in S},\delta_i\bigg) \bigg(\big(r_{g(\kappa)}^{s}(g)-d_\kappa^{\gamma_i^{-1}s}(g)\big)_{s\in S},\gamma_i\bigg)\\
&= \rho_{g(\kappa)}(h)\rho_\kappa(g) \text{.}
\end{align*}
\end{proof}

Note that we are still assuming $SV_G$ is finitely presented, hence finitely generated, so by \cite[Theorem~A]{belk22} we know $G$ is finitely generated and $S$ has finitely many $G$-orbits, which also tells us that $\Z\wr_S G$ is finitely generated (see for example \cite[Proposition~2.1]{cornulier06}). At this point we fix some finite symmetric generating set $A$ for $SV_G$. (Here a subset of a group is \emph{symmetric} if it is closed under taking inverses.)

\begin{lemma}\label{lem:good_gen_sets}
There exists a finite symmetric generating set $B$ for $\Z\wr_S G$ such that $B$ contains $\{\rho_\kappa(a)\mid \kappa\in\Cantor^S,a\in A\}$.
\end{lemma}

\begin{proof}
Since we already know $\Z\wr_S G$ is finitely generated, all we have to do is prove that $\{\rho_\kappa(a)\mid \kappa\in\Cantor^S,a\in A\}$ is finite. Indeed, we claim that for any $h\in SV_G$ the set $\{\rho_\kappa(h)\mid \kappa\in\Cantor^S\}$ is finite. This is because, fixing a domain and range partition that encode $h$, if $\kappa$ and $\kappa'$ share a domain block then by construction we have $\rho_\kappa(h)=\rho_{\kappa'}(h)$. Since our fixed domain partition has finitely many blocks, $\{\rho_\kappa(h)\mid \kappa\in\Cantor^S\}$ is finite. Now since $A$ is finite, $\{\rho_\kappa(a)\mid \kappa\in\Cantor^S,a\in A\}$ is finite.
\end{proof}

We will use \emph{left} word metrics with respect to the finite generating sets $A$ and $B$. Now we want to show that $\rho_{\kappa_0}$ is a quasi-retraction with respect to these word metrics, where $\kappa_0$ is the point satisfying $\kappa_0(s)=\overline{0}$ for all $s\in S$.

\begin{proposition}\label{prop:qr}
The function $\rho_{\kappa_0}$ is a quasi-retraction.
\end{proposition}

\begin{proof}
First we need to construct a function $\zeta \colon \Z\wr_S G \to SV_G$. Embed $\Z$ into $V$ by sending $1$ to some element that acts by sending the dyadic brick $B(0)$ in $\Cantor$ to $B(00)$ via $0\kappa\mapsto 00\kappa$. Extend this to $\bigoplus_S \Z \to \bigoplus_S V$. Now observe that $\bigoplus_S V$ embeds into $SV$ by having the copy of $V$ corresponding to $s\in S$ act on the factor of $\Cantor$ in $\Cantor^S$ corresponding to $s$, and fixing all other coordinates. Finally, $G$ embeds into $SV_G$ as a group of global coordinate permutations; following \cite{belk22} we write $\iota_\varnothing\colon G\to SV_G$ for this embedding. Clearly this copy of $G$ normalizes the copy of $\bigoplus_S V$, and acts by permuting coordinates, so together they form a copy of $V\wr_S G$. Putting everything together we get a monomorphism $\zeta \colon \Z\wr_S G \to SV_G$.

Now we claim that $\rho_{\kappa_0}\circ\zeta$ is the identity on $\Z\wr_S G$. Note that every element in the image of $\zeta$ fixes $\kappa_0$, so restricted to this image, $\rho_{\kappa_0}$ is a homomorphism by Lemma~\ref{lem:rho_prod}. Thus, it suffices to check that $\rho_{\kappa_0}\circ\zeta$ is the identity on both the $\bigoplus_S \Z$ and $G$ factors individually, or even just on some generating sets thereof. On the $G$ factor, $\zeta(\gamma)=\iota_\varnothing(\gamma)$, which $\rho_{\kappa_0}$ sends to $\gamma$ as desired. Now let $(z_s)_{s\in S}\in \bigoplus_S \Z$, and assume that it comes from the standard generating set, so $z_s=1$ for some $s\in S$ and $z_t=0$ for all $t\ne s$. Let $\varphi\colon S\to \{0,1\}^*$ send $s$ to $0$ and all other $t$ to $\varnothing$. Restricted to $B(\varphi)$, the element $h=\zeta((z_s)_{s\in S},1_G)$ acts in the $s$ dimension by prepending a zero and acts trivially in all other dimensions. Since $B(\varphi)$ contains $\kappa_0$, and it is possible to find for $h$ a domain partition into dyadic bricks including $B(\varphi)$ and a range partition into dyadic bricks including the image of $B(\varphi)$ under $h$, this shows that $r_{\kappa_0}^s(h) - d_{\kappa_0}^s(h)=1$ and $r_{\kappa_0}^t(h) - d_{\kappa_0}^t(h)=0$ for all $t\ne s$.We conclude that $\rho_{\kappa_0}(h)=((z_s)_{s\in S},1_G)$ as desired.

Finally, we need to prove that $\rho_{\kappa_0}$ and $\zeta$ are coarse Lipschitz, with respect to the generating sets $A$ and $B$. This is immediate for $\zeta$ since it is a homomorphism, so we just need to handle $\rho_{\kappa_0}$. In fact we will show that every $\rho_\kappa$ is non-expanding, i.e., $d(\rho_{\kappa}(h),\rho_{\kappa}(h'))\le d(h,h')$ for all $h,h'\in SV_G$. It suffices to do this in the case when $h'=ah$ for $a\in A$, i.e., to prove that $d(\rho_{\kappa}(h),\rho_{\kappa}(ah))\le 1$. Indeed, $\rho_{\kappa}(ah) = \rho_{h(\kappa)}(a)\rho_\kappa(h)$ by Lemma~\ref{lem:rho_prod}, and $\rho_{h(\kappa)}(a)\in B$ by Lemma~\ref{lem:good_gen_sets}.
\end{proof}

Now we can prove the main result of this section.

\begin{proof}[Proof of Proposition~\ref{prop:forward}]
Since $\Z\wr_S G$ is a quasi-retract of $SV_G$ by Proposition~\ref{prop:qr}, we have that $\Z\wr_S G$ is finitely presented. Thus, by Citation~\ref{cit:cornulier}, the action of $G$ on $S$ is of type~(A).
\end{proof}

\begin{remark}
This also shows that if $SV_G$ is of type $\F_n$ then so is $\Z\wr_S G$, which imposes restrictions on $G$ and $S$, thanks to \cite[Theorem~B]{bartholdi15}. Namely, we must have that $G$ is of type $\F_n$, each $\Stab_G(T)$ for $T\subseteq S$ with $|T|<n$ is of type $\F_{n-|T|}$, and there are finitely many orbits of $n$-element subsets of $S$. Let us call the action of $G$ on $S$ in this case an action of \emph{type~(A$_n$)}, so type~(A$_2$) means type~(A). This is conjectured to be an if and only if in \cite[Conjecture~H]{belk22}, so it is worth remarking here that our quasi-retract argument in this section has now proven the ``only if'' direction of this conjecture (and the main result of the present paper is that the ``if'' direction is also true for $n=2$). It remains open whether the action of $G$ on $S$ being of type~(A$_n$) implies $SV_G$ is of type $\F_n$.
\end{remark}

%-------------------------------------------------
\section{Finite presentability from actions with bad edge stabilizers}\label{sec:fp}

It is a classical fact that if a group $\Gamma$ acts cellularly and cocompactly on a simply connected CW-complex, with every vertex stabilizer finitely presented and every edge stabilizer finitely generated, then $\Gamma$ is finitely presented. See for example \cite[Proposition~3.1]{brown87}. If not every edge stabilizer in $\Gamma$ is finitely generated, then this is an impediment to obtaining finite presentability of $\Gamma$. In this section we prove a technical result that provides a way of getting around this impediment (somewhat literally).

Let us first recall how a group presentation falls out of a cellular action on a simply connected complex. In \cite{brown84} the complex can be any CW-complex and the action is allowed to stabilize cells without fixing them. For simplicity, here we will follow \cite{armstrong88} and assume the complex is simplicial and the action is \emph{rigid}, meaning the stabilizer of each cell equals its fixer. See also \cite{putman11} for a nice simplification in the case when the orbit space is $2$-connected (which is unfortunately not the case for the situations we will care about later). All this coming discussion before Proposition~\ref{prop:technical} is taken directly from \cite{armstrong88}.

\medskip

Let $K$ be a simply connected simplicial complex. View the $1$-skeleton $K^{(1)}$ as a (simplicial) graph. As in \cite{armstrong88}, following Serre \cite{serre77}, each edge $e$ can be made into a pair of directed edges by specifying which endpoint is the origin $o(e)$ and which is the terminus $t(e)$. Write $\overline{e}$ for $e$ with the opposite orientation, so $o(\overline{e})=t(e)$ and $t(\overline{e})=o(e)$, and both $e$ and $\overline{e}$ are viewed as (directed) edges in $K^{(1)}$. Let $\Gamma$ be a group with an orientation preserving action on $K$, i.e., for any edge $e$ and any $\gamma\in \Gamma$ we have $o(\gamma.e)=\gamma.o(e)$ and $t(\gamma.e)=\gamma.t(e)$. Let $M$ be a maximal tree in the orbit space $\Gamma\backslash K^{(1)}$, and let $T$ be a choice of lift of $M$ to a subtree of $K^{(1)}$. Note that the vertices of $T$ form a set of representatives of the $\Gamma$-orbits of vertices of $K$. For each edge $f$ of $(\Gamma\backslash K^{(1)})\setminus M$, choose a lift $\widetilde{f}$ of $f$ in $K^{(1)}$ such that one endpoint of $\widetilde{f}$ lies in $T$, call it the origin $o(\widetilde{f})$. The other endpoint of $\widetilde{f}$, its terminus $t(\widetilde{f})$, necessarily lies outside $T$. Let $z_f$ be the vertex of $T$ that shares a $\Gamma$-orbit with $t(\widetilde{f})$, and choose some $\gamma_f\in \Gamma$ such that $\gamma_f.z_f=t(\widetilde{f})$. Note that for each $f$ we are free to choose $\gamma_f$ to be any element we like, as long as it satisfies $\gamma_f.z_f=t(\widetilde{f})$. We will also always choose the lift $\widetilde{(\overline{f})}$ of $\overline{f}$ to have $z_f$ specifically as its origin, so consequently $z_{\overline{f}}$ is the origin of $\widetilde{f}$, and we can take $\gamma_{\overline{f}}=\gamma_f^{-1}$. For each edge $f$ of $M$, set $\gamma_f=1$.

Let us say \emph{triangle} for $2$-simplex. From each orbit of triangles in $K^{(2)}$, we can choose one that has at least one vertex in $T$. Given a triangle $\Delta$ in $K^{(2)}$ with a vertex $v$ in $T$, view the boundary of $\Delta$ as an edge path (of length three) from $v$ to $v$, say $e_1,e_2,e_3$. (Up to possibly switching the roles of some $e_i$ and $\overline{e}_i$, we can assume that this is a directed cycle, i.e., $t(e_i)=o(e_{i+1})$ for all $i$ mod $3$.) For each $i=1,2,3$, let $f_i$ be the image of $e_i$ in $\Gamma\backslash K^{(1)}$, and let $\widetilde{f}_i$ be the lift as above. If $f_i$ is in $M$ then $\widetilde{f}_i$ is in $T$, and if not then the origin $o(\widetilde{f}_i)$ lies in $T$ but not the terminus. Since $e_1$ and $\widetilde{f}_1$ share a $\Gamma$-orbit, and both of their origins lie in $T$, they must have the same origin, namely $v$. Since they share an orbit and have the same origin $v$, and the action is orientation preserving, we can choose $a_1\in \Gamma_v\coloneqq \Stab_\Gamma(v)$ taking $\widetilde{f}_1$ to $e_1$. Now observe that $a_1 \gamma_{f_1}$ takes $z_{f_1}$ to the terminus of $e_1$, which is the origin of $e_2$. Since the origin of $\widetilde{f}_2$ is the vertex of $T$ sharing an orbit with the origin of $e_2$, the origin of $\widetilde{f}_2$ is $z_{f_1}$. Now similar to before, we can choose $a_2\in \Gamma_{z_{f_1}}$ such that $a_1 \gamma_{f_1} a_2$ takes $\widetilde{f}_2$ to $e_2$, and so $a_1 \gamma_{f_1} a_2 \gamma_{f_2}$ takes $z_{f_2}$ to the terminus of $e_2$, which is the origin of $e_3$. Finally, we do this trick again and get $a_3\in \Gamma_{z_{f_2}}$ such that $a_1 \gamma_{f_1} a_2 \gamma_{f_2} a_3 \gamma_{f_3}$ takes $z_{f_3}$ to the terminus of $e_3$, which is $v$. But $z_{f_3}$ is a vertex in $T$ sharing an orbit with $v$, hence must be $v$. We conclude that $a_1 \gamma_{f_1} a_2 \gamma_{f_2} a_3 \gamma_{f_3}$ lies in $\Gamma_v$, write $a_1 \gamma_{f_1} a_2 \gamma_{f_2} a_3 \gamma_{f_3}=a_v$. Now we have $a_v^{-1}a_1 \gamma_{f_1} a_2 \gamma_{f_2} a_3 \gamma_{f_3} = 1$, and up to forgetting how $a_1$ arose in $\Gamma_v$ specifically taking $\widetilde{f}_1$ to $e_1$ we can rename $a_v^{-1}a_1$ as the new $a_1$. Thus we have a relation
\[
a_1 \gamma_{f_1} a_2 \gamma_{f_2} a_3 \gamma_{f_3} = 1\text{,}
\]
with $a_1$, $a_2$, and $a_3$ coming from various vertex stabilizers, and we call this relation $r_\Delta$. We can phrase $r_\Delta$ as saying that $a_1 \gamma_{f_1} a_2 \gamma_{f_2} a_3 = \gamma_{f_3}^{-1}$, so in particular if we have already chosen $\gamma_{f_1}$ and $\gamma_{f_2}$, then we can choose $\gamma_{f_3}$ in terms of $\gamma_{f_1}$, $\gamma_{f_2}$, and elements of the various vertex stabilizers.

\medskip

With this setup, we get a presentation for $\Gamma$ (see the Presentation Theorem at the end of Section~2 of \cite{armstrong88}). The set of generators is the (disjoint) union of the vertex stabilizers $\Gamma_v$ for $v$ a vertex of $T$, together with a generator $\lambda_f$ for each $f$ an edge of $\Gamma\backslash K^{(1)}$. If an element $g$ lies in more than one $\Gamma_v$, write $g_v$ when viewing it specifically as an element of $\Gamma_v$. The defining relations consist of the relations in each $\Gamma_v$, called \emph{vertex relations}, the set of relations
\[
r_\Delta^\lambda \quad\colon\quad a_1 \lambda_{f_1} a_2 \lambda_{f_2} a_3 \lambda_{f_3} = 1
\]
obtained by replacing each $\gamma$ with $\lambda$ in $r_\Delta$, called \emph{triangle relations}, the relations $\lambda_f=1$ for all $f$ in $M$, called \emph{tree relations}, and the relations
\[
R(f,g) \quad\colon\quad (g_{o(\widetilde{f})})^{\lambda_f} = (g^{\gamma_f})_{z_f}
\]
for all $f$ in $(\Gamma\backslash K^{(1)})\setminus M$ and $g\in \Gamma_{\widetilde{f}}$, called \emph{edge relations}. Here as usual $x^y\coloneqq y^{-1}xy$. (To be more clear about $\gamma$ versus $\lambda$, the $\gamma_f$ are elements of the group $\Gamma$, whereas the $\lambda_f$ are formal symbols used as generators in the abstract group presentation. The tree relations $\lambda_f=1$ essentially account for the fact that we chose $\gamma_f=1$ whenever $f$ in the tree $M$.)

For a fixed $f$, write $R(f,*)$ for the family of relations $R(f,g)$ for $g\in \Gamma_{\widetilde{f}}$. Note that all the $R(f,g)$ are consequences of those for which $g$ comes from a generating set of $\Gamma_{\widetilde{f}}$. At this point it is clear that if each $\Gamma_v$ is finitely presented and each $\Gamma_{\widetilde{f}}$ is finitely generated, and the action is cocompact, then $\Gamma$ is finitely presented. Now we will inspect the case when not every $\Gamma_{\widetilde{f}}$ is finitely generated.

\begin{proposition}\label{prop:technical}
Let $K$ be a simply connected simplicial complex and $\Gamma$ a group with an orientation preserving, cocompact action on $K$, such that every vertex stabilizer $\Gamma_v$ is finitely presented. Suppose that for each edge $e=\{v,w\}$, there exists an edge path $e_1,\dots,e_n$ from $v$ to $w$ such that each stabilizer $\Gamma_{e_i}$ is finitely generated, and the subgroup $\bigcap\limits_{i=1}^n \Gamma_{e_i}$ of $\Gamma_e$ has finite index. Then $\Gamma$ is finitely presented.
\end{proposition}

\begin{proof}
Let $f$ be an edge in $\Gamma\backslash K^{(1)}$ such that $\Gamma_{\widetilde{f}}$ is not finitely generated. Let $e_1,\dots,e_n$ be an edge path in $K^{(1)}$ from $v=o(\widetilde{f})$ to $w=t(\widetilde{f})$ such that each $\Gamma_{e_i}$ is finitely generated, and $\bigcap\limits_{i=1}^n \Gamma_{e_i}$ has finite index in $\Gamma_{\widetilde{f}}$. Let $f_i$ be the image of $e_i$ in $\Gamma\backslash K^{(1)}$, with $\widetilde{f}_i$ the chosen lift to $K^{(1)}$ (so $o(\widetilde{f}_i)$ lies in $T$). We now claim that the edge relations $R(f,*)$ are all consequences of the edge relations $R(f_i,*)$ together with the (finitely many) non-edge relations. By extrapolating the triangle relations, and noting that the path $e_1,\dots,e_n,\overline{f}$ is a directed cycle, we see that we are free to choose $\gamma_f$ (which equals $\gamma_{\overline{f}}^{-1}$) so that
\[
\gamma_f=a_1\gamma_{f_1}a_2\gamma_{f_2}\cdots a_n \gamma_{f_n} a_{n+1} \text{,}
\]
for some $a_1,\dots,a_{n+1}$ coming from the various stabilizers of vertices of $T$. From the triangle relations, we also have
\[
\lambda_f=a_1\lambda_{f_1}a_2\lambda_{f_2}\cdots a_n \lambda_{f_n} a_{n+1} \text{.}
\]
Now for any $g$ in the finite index subgroup of $\Gamma_{\widetilde{f}}$ fixing the path $e_1,\dots,e_n$, the relation $R(f,g)$, which is $(g_v)^{\lambda_f}=(g^{\gamma_f})_{z_f}$, can be rewritten as
\[
(g_v)^{a_1\lambda_{f_1}a_2\lambda_{f_2}\cdots a_n \lambda_{f_n} a_{n+1}}=(g^{a_1\gamma_{f_1}a_2\gamma_{f_2}\cdots a_n \gamma_{f_n} a_{n+1}})_{z_f} \text{.}
\]
This shows that $R(f,g)$ is a consequence of vertex relations and edge relations from the $R(f_i,*)$ (and tree relations if needed). Since this holds for all $g$ coming from a finite index subgroup of $\Gamma_{\widetilde{f}}$, we can choose some (finite) set of representatives in $\Gamma_{\widetilde{f}}$ of the cosets of this subgroup, and conclude that all the $R(f,g)$ are consequences of the vertex relations, triangle relations, tree relations, edge relations from the $R(f_i,*)$, and finitely many edge relations from $R(f,*)$. Doing this for every such $f$ proves that all the defining relations are consequences of a finite set of relations, i.e., $\Gamma$ is finitely presented.
\end{proof}

It seems likely that some sort of higher dimensional analog of this result is also true, and could be used for deducing that groups admitting certain actions are of type $\F_n$, perhaps using spectral sequence techniques. However, some results in the next section, specifically Lemma~\ref{lem:long_edge_stabs}, do not have clear higher dimensional analogs, and so we will not pursue this further here.

%-------------------------------------------------
\section{Stein complexes and subcomplexes}\label{sec:stein}

In this section we recall the construction of the Stein complex $X$ for $SV_G$, and establish some important subcomplexes of $X$. Given two elements $h$ and $h'$ of $S\mathcal{V}_G$, say $h$ has rank $n$ and corank $m$, and $h'$ has rank $n'$ and corank $m'$, define the \emph{direct sum} $h\oplus h'$ to be the element of $S\mathcal{V}_G$ with rank $n+n'$ and corank $m+m'$ given by sending the first $m$ cubes of $\Cantor^S(m+m')$ to the first $n$ cubes of $\Cantor^S(n+n')$ via $h$, and sending the last $m'$ cubes of $\Cantor^S(m+m')$ to the last $n'$ cubes of $\Cantor^S(n+n')$ via $h'$.

Given a permutation $\sigma$ in the symmetric group $\Sigma_m$, the \emph{permutation homeomorphism} $p_\sigma$ of $\Cantor^S(m)$ is defined by sending the $i$th Cantor cube to the $\sigma(i)$th Cantor cube via the canonical homeomorphism. By a \emph{twisted permutation} we will mean any composition of a permutation homeomorphism with a direct sum of twist homeomorphisms. It is easy to see that the twisted permutations with rank (and corank) $m$ form a group isomorphic to $G\wr_m \Sigma_m \coloneqq \bigoplus_{i=1}^m G \rtimes \Sigma_m$; we denote this group by $\mathcal{G}(m)$.

Next, define a \emph{multicolored tree} recursively, by saying that the identity on $\Cantor^S$ and all simple splits $x_s$ are multicolored trees, and if $f_1$ and $f_2$ are multicolored trees and $\sigma$ is a permutation, then $f=p_\sigma(f_1\oplus f_2)x_s$ is a multicolored tree. The name comes from viewing $S$ as a set of colors, and $x_s$ as a caret with color $s$. Note that multicolored trees have corank $1$, and can have any rank. A \emph{multicolored forest} is any element of $S\mathcal{V}$ of the form $p_\sigma(f_1\oplus\cdots\oplus f_m)$ for multicolored trees $f_i$ and $\sigma$ a permutation. Note that the set of all multicolored forests is the smallest subset of $S\mathcal{V}$ that contains all permutations and simple splits, and is closed under direct sums and compositions.

\medskip

Let $P$ be the set of equivalence classes $[h]=\mathcal{G}(m)h$, where $h$ is an element of $S\mathcal{V}_G$ with rank $m$. That is, $P$ is the set of elements of the twisted Brin--Thompson groupoid, up to postcomposing by twisted permutations. Let $P_1$ be the subset of elements with corank $1$. We have a (right) action of $SV_G$ on $P_1$ given by precomposing. The set $P$ has a partial order $\le$, given by
\[
[h]\le [fh]
\]
whenever $f$ is a multicolored forest (whose corank equals the rank of $h$). Note that $P_1$ is a subposet of $P$. By \cite[Lemma~5.1 and Proposition~5.2]{belk22}, $\le$ is a partial order and the subposet $P_1$ is directed, so the geometric realization $|P_1|$ is contractible. One key to $\le$ being a partial order is the following result, which will also be important here.

\begin{cit}\cite[Lemma~2.3]{belk22}\label{lem:swap}
If $f$ is a multicolored forest and $g$ is a twisted permutation such that the composition $fg$ is defined, then $fg = g'f'$ for some twisted permutation $g'$ and some multicolored forest $f'$. Moreover, if $f$ is a multicolored tree then $g'$ is a direct sum of copies of $g$.
\end{cit}

In \cite{belk22}, a certain subcomplex $X$ of $|P_1|$ called the Stein complex was constructed, using a notion of ``elementary'' multicolored forests, to provide a smaller and more manageable, but still contractible, complex for $SV_G$ to act on. Our goal now is to find an even smaller subcomplex that is even more manageable, and is still simply connected (if not contractible). This will be enough to lead to finite presentability of $SV_G$, with weaker hypotheses.

\begin{definition}[Elementary, spectrum]
We will recall what it means for a multicolored tree to be \emph{elementary}, and at the same time recall the notion of \emph{spectrum}, recursively. First, the identity on $\Cantor^S$ is declared to be elementary with empty spectrum. Given elementary multicolored trees $f_1$ and $f_2$ with spectra $S_1,S_2\subseteq S$, if $s\in S\setminus(S_1\cup S_2)$ then we declare that the multicolored tree $p_\sigma(f_1\oplus f_2)x_s$ (for any $\sigma$) is elementary and has spectrum $S_1\cup S_2\cup\{s\}$.
\end{definition}

Intuitively, the spectrum of a multicolored tree is the set of colors ``involved'' in the tree, and a tree is elementary if no color occurs more than once going from the root to a leaf. Note that it is fine if $S_1$ and $S_2$ intersect. We extend these definitions to multicolored forests in the obvious way: a multicolored forest is \emph{elementary} if it is a direct sum of elementary multicolored trees, composed with a permutation homeomorphism, and its \emph{spectrum} is the union of the spectra of these elementary multicolored trees.

\medskip

Now we introduce a new, crucial definition.

\begin{definition}[$k$-elementary]
Call an elementary multicolored forest \emph{$k$-elementary} if its spectrum has size at most $k$.
\end{definition}

For example, the only $1$-elementary multicolored forests are the direct sums of copies of identity elements and copies of some simple split $x_s$, composed with a permutation homeomorphism. As another example, the multicolored tree $((x_u\oplus x_t))x_s$ is $3$-elementary if $s$, $t$, and $u$ are distinct, $2$-elementary if $s\ne t=u$, and non-elementary if $s\in\{t,u\}$.

For $[h]\preceq [fh]$ in $P_1$, write $[h]\preceq_k [fh]$ whenever the multicolored forest $f$ is $k$-elementary. Given a simplex in $|P_1|$, i.e., a chain $v_0<v_1<\cdots<v_k$ of elements of $P_1$, call the simplex \emph{$k$-elementary} if $v_0\preceq_k v_k$ (and hence $v_i\preceq_k v_j$ for all $i<j$). These simplices form a subcomplex of $|P_1|$, which we denote by $X(k)$, and call the \emph{$k$-elementary Stein complex} for $SV_G$. The \emph{Stein complex} $X$ is $X=X(\infty)$, i.e., the complex dictated by using all elementary multicolored forests rather than restricting to just the $k$-elementary ones for some $k$.

Let $\phi \colon X^{(0)} \to \N$ be the function sending $v$ to its rank, and also denote by $\phi$ the map $X\to\R$ given by affinely extending $\phi$ to each simplex. For each $m\in\N$ let $X_m$ be the full subcomplex of $X$ spanned by vertices with rank at most $m$. (In \cite{belk22} $X_m$ is defined to be $\phi^{-1}([1,m])$, but this is homotopy equivalent to what we are calling $X_m$, thanks to standard discrete Morse theory, for example \cite[Lemma~2.5]{bestvina97}. We do things this way since it will be convenient for us to have $X_m$ be a subcomplex.) Call $X_m$ the \emph{truncation} of $X$ at $m$. Since $\phi$ is invariant under the action of $SV_G$ on $X$, each $X_m$ is stabilized by $SV_G$. Write $X_m(k)$ for the intersection
\[
X_m(k) \coloneqq X_m \cap X(k) \text{.}
\]
By \cite[Proposition~7.9]{belk22}, which establishes connectivity properties of the descending links, together with standard discrete Morse theory, e.g., \cite[Corollary~2.6]{bestvina97}, one can compute that for any $n$, the complex $X_m$ is $(n-1)$-connected for $m$ sufficiently large. For example, $X_m$ is simply connected for all $m\ge 21$, since \cite[Proposition~7.9]{belk22} says the descending link of any vertex with rank greater than $m$ is simply connected as soon as $\lfloor((m+1)/2-2)/3\rfloor-2\ge 1$ and $\log_2((m+1)/2)-2\ge 1$, which is equivalent to $m\ge 21$.

\begin{proposition}\label{prop:hi_conn}
Let $k\ge n$. The complex $X(k)$ is $(n-1)$-connected, and moreover for any $m\in\N$, if $m$ is large enough that the truncation $X_m$ is $(n-1)$-connected then $X_m(k)$ is also $(n-1)$-connected. For example, $X_{21}(2)$ is simply connected.
\end{proposition}

\begin{proof}
We will begin by mimicking part of the proof of \cite[Proposition~5.6]{belk22}. By that proposition, the Stein complex $X$ is contractible. We can build up from $X(k)$ to $X$ by attaching the geometric realizations of closed poset intervals of the form $[v,w]$ for $v<w$ satisfying $v\preceq w$ but $v\not\preceq_k w$. We do so in increasing order of the rank of $w$ minus the rank of $v$, so at the point when we attach $|[v,w]|$, we do so along an intersection equal to $|[v,w)\cup(v,w]|$. This is the suspension of $|(v,w)|$. In order to conclude that $X(k)$ is $(n-1)$-connected, it therefore suffices to prove that the suspension of $|(v,w)|$ is $(n-1)$-connected, i.e., that $|(v,w)|$ is $(n-2)$-connected. Since $k\ge n$, it suffices to show that $|(v,w)|$ is $(k-2)$-connected.

Say $v=[h]$ and $w=[fh]$ for $f$ an elementary (and not $k$-elementary) multicolored forest. Let $r$ be the rank of $h$, which is also the corank of $f$. For any $u\in (v,w)$, since $v<u$ we can write $u=[f'h]$ for some non-trivial multicolored forest $f'$ (here we have used Lemma~\ref{lem:swap}), so in particular there exists a simple split $x_s$ such that $v<[(\id_p \oplus x_s \oplus \id_{r-p-1}) h]\le u<w$ for some $0\le p\le r-1$. Here $\id_p$ is the identity element with rank (and corank) $p$. In particular, every simplex of $|(v,w)|$ lies in the star of a vertex of $|(v,w)|$ of the form $u_p(s)\coloneqq [(\id_p \oplus x_s \oplus \id_{r-p-1}) h]$. The complex $|(v,w)|$ is therefore covered by these stars, and we can use a Nerve Lemma, for example \cite[Lemma~1.2]{bjoerner94}, to inspect its connectivity. Note that the distinguishing feature of these vertices $u_p(s)$ is that $v\preceq_1 u_p(s)$.

To simplify notation, we can assume without loss of generality that $h=\id_r$, so all the vertices in $|(v,w)|$ are represented by (elementary) multicolored forests. By \cite[Proposition~5.3]{belk22}, any two elements of the poset $P$ represented by (elementary) multicolored forests $[f_1],[f_2]$ with the same corank have a unique least upper bound, their \emph{join} $[f_1]\vee[f_2]$, which is again represented by an (elementary) multicolored forest. It is also clear from the proof of \cite[Proposition~5.3]{belk22} that if $[f_3]=[f_1]\vee[f_2]$ then the spectrum of $f_3$ equals the union of the spectra of $f_1$ and $f_2$. In particular, for any $u_{p_1}(s_1),\dots,u_{p_\ell}(s_\ell)$ in $|(v,w)|$, the join of these elements is represented by an elementary multicolored forest with spectrum of size at most $\ell$. Since $f$ is not $k$-elementary, whenever $\ell\le k$ we have that this join lies in $(v,w)$, and so we conclude that for any $1\le \ell\le k$ and any $u_{p_1}(s_1),\dots,u_{p_\ell}(s_\ell)$ in $|(v,w)|$, the stars of these $u_{p_i}(s_i)$ have contractible intersection, namely the star of the join of $u_{p_1}(s_1),\dots,u_{p_\ell}(s_\ell)$. This shows that the $(k-1)$-skeleton of the nerve of the covering by all these stars equals the $(k-1)$-skeleton of a simplex, hence is $(k-2)$-connected. (The dimension of the simplex is one less than the total number of such stars, which could be less than $k-1$, in which case the nerve is the entire simplex, hence contractible.) This also shows that the $(k-1)$-skeleton of $|(v,w)|$ has the same connectivity as the $(k-1)$-skeleton of the nerve \cite[Lemma~1.2]{bjoerner94}, and so we conclude that $|(v,w)|$ is $(k-2)$-connected, and hence $(n-2)$-connected, as desired.

Now consider the truncation $X_m(k)$ for $m$ large enough that $X_m$ is $(n-1)$-connected. We can do the exact same procedure as above since all the vertices in the open interval $(v,w)$ have rank less than the rank of $w$, hence less than $m$. We conclude that $X_m(k)$ is $(n-1)$-connected for all $k\ge n$ and all sufficiently large $m$.
\end{proof}

\medskip

The action of $SV_G$ on $|P_1|$ stabilizes $X_m(k)$ for each $m$ and $k$, so we can inspect the action of $SV_G$ on $X_m(k)$. If $m$ and $k$ are large enough, this is simply connected, for example $X_{21}(2)$. In order to get $SV_G$ to be finitely presented, we need to know things about cocompactness and about stabilizers. First we establish cocompactness

\begin{lemma}[Cocompact]\label{lem:cocpt}
If the action of $G$ on $S$ has finitely many orbits of $k$-element subsets of $S$, then the action of $SV_G$ on $X_m(k)$ is cocompact, for each $m\in\N$.
\end{lemma}

\begin{proof}
We follow the proof of \cite[Proposition~6.7]{belk22}. There are finitely many $SV_G$-orbits of vertices in $X_m(k)$, since $SV_G$ is transitive on vertices of a given rank, and there are only finitely many ranks between $1$ and $m$. Now we need to show that for each vertex $v$ of rank $r\le m$, there are finitely many $\Stab_{SV_G}(v)$-orbits of simplices in $X_m(k)$ having $v$ as their vertex of minimum rank. Since each closed interval $[v,w]$ is finite, it is enough to check that there are finitely many $\Stab_{SV_G}(v)$-orbits of edges of the form $v<w$ in $X_m(k)$, or equivalently finitely many $\Stab_{SV_G}(v)$-orbits of vertices $w$ satisfying $v\preceq_k w$. Without loss of generality $v=[\id_r]$, so we want to show that there are finitely many $\mathcal{G}(r)$-orbits of vertices of the form $[f]$ for $f$ a $k$-elementary multicolored forest with rank at most $m$. Since $f$ is $k$-elementary and $S$ has finitely many $G$-orbits of $k$-element subsets, we can fix some $k$-element subset $T$ of $S$ such that every such $\mathcal{G}(r)$-orbit contains an element $[f]$ such that the spectrum of $f$ lies in $T$. The result now follows from the fact that there are only finitely many multicolored forests with rank at most $m$ and spectrum lying in $T$.
\end{proof}

We can also establish some results about stabilizers. Recall that two groups are \emph{commensurable} if they have isomorphic finite index subgroups. Note that any group commensurable to a group of type $\F_n$ is also of type $\F_n$, and a direct product of finitely many groups of type $\F_n$ is of type $\F_n$. This is because, on the level of classifying spaces, a complex has finite $n$-skeleton if and only if every finite-sheeted cover does, and a product of complexes with finite $n$-skeleta has finite $n$-skeleton.

\begin{cit}\cite[Lemma~6.3]{belk22}\label{cit:vtx_stab}
Let $v$ be a vertex of $X$ with rank $m$. Then the stabilizer of $v$ in $SV_G$ is isomorphic to $G\wr_m \Sigma_m$, hence commensurable to $G^m$. In particular, if $G$ is of type $\F_n$ then so is this stabilizer.
\end{cit}

\begin{definition}[Short/long edges]
Let $[h]<[fh]$ be an edge in $X$, so $f$ is an elementary multicolored forest. Call this edge \emph{short} if $f$ is a direct sum of one simple split and some number of copies of the identity, composed with a permutation homeomorphism. In other words, this edge is short provided that the rank of $[fh]$ is only one more than that of $[h]$. Call an edge \emph{long} if it is not short.
\end{definition}

The following is immediate from \cite[Proposition~6.5]{belk22}, since the spectrum of a simple split has size one.

\begin{cit}[Short edge stabilizers]\label{cit:short_edge_stab}
The stabilizer of any short edge in $SV_G$ is commensurable to a direct sum of finitely many copies of $G$ and $\Stab_G(s)$ for $s\in S$. In particular, if $G$ and each $\Stab_G(s)$ are of type $\F_n$, then so is the stabilizer of any short edge.
\end{cit}

\begin{lemma}[Long edge stabilizers]\label{lem:long_edge_stabs}
Let $v<w$ be a long edge. Then there is a path of short edges from $v$ to $w$ such that the fixer in $SV_G$ of this path has finite index in the stabilizer of $v<w$.
\end{lemma}

\begin{proof}
Since the closed interval $[v,w]$ is finite, the fixer of this entire interval has finite index in the stabilizer of $v<w$. Clearly there is a path of short edges from $v$ to $w$ lying in $|[v,w]|$, so the fixer of this path has finite index in the stabilizer of $v<w$.
\end{proof}

Note that even if every short edge stabilizer is finitely generated, this does not imply every long edge stabilizer is too. Indeed, in this case the fixer of a path of short edges is an intersection of finitely generated subgroups, which has no reason to also be finitely generated. However, we can use Proposition~\ref{prop:technical} to get around this issue, and are now poised to prove Theorem~\ref{thrm:main}.

\begin{proof}[Proof of Theorem~\ref{thrm:main}]
The forward direction is Proposition~\ref{prop:forward}. Now we do the backward direction. Suppose $G$ is a group with an action of type~(A) on a set $S$, and we need to prove that $SV_G$ is finitely presented. Consider the (orientation preserving) action of $SV_G$ on the complex $X_m(2)$, for $m\ge 21$. This complex is simply connected by Proposition~\ref{prop:hi_conn}, and the action is cocompact by Lemma~\ref{lem:cocpt}. Every vertex stabilizer is finitely presented by Citation~\ref{cit:vtx_stab}. Any stabilizer of a short edge is finitely generated by Citation~\ref{cit:short_edge_stab}. Finally, for any long edge, by Lemma~\ref{lem:long_edge_stabs} there is a path of short edges connecting its endpoints, such that the fixer of the path has finite index in the stabilizer of the long edge. Now Proposition~\ref{prop:technical} tells us that $SV_G$ is finitely presented.
\end{proof}

It would be interesting to try and prove some sort of higher dimensional analog of Lemma~\ref{lem:long_edge_stabs}, but it is unclear what to hope for. The crucial difference between the 1-dimensional versus higher dimension situations is that even if an edge is ``long'', its proper faces are all ``short'' (being vertices), whereas for whatever notion of ``long'' we try for higher simplices, proper faces will likely also be able to be long.

%-------------------------------------------------
\section{Applications to the Boone--Higman conjecture}\label{sec:bh}

Recall Corollary~\ref{cor:main}, that any subgroup of a group admitting an action of type~(A) (has solvable word problem and) satisfies the Boone--Higman conjecture. In this section, we investigate this further.

\subsection{Permutational Boone--Higman}\label{ssec:bh_a}

First, let us define the following relative of the Boone--Higman conjecture:

\begin{conjecture}[Permutational Boone--Higman conjecture]
A finitely generated group has solvable word problem if and only if it embeds in a group admitting an action of type~(A).
\end{conjecture}

The ``if'' direction is clear from the ``if'' direction of the Boone--Higman conjecture together with Theorem~\ref{thrm:main}. When we say a group ``satisfies'' this conjecture, we mean it (has solvable word problem and) embeds in a group admitting an action of type~(A). We do not know whether a group can satisfy the Boone--Higman conjecture without satisfying the permutational version (see Question~\ref{quest:BHC_PBHC}). Let us also emphasize that when we say a group satisfies the permutational Boone--Higman conjecture, we do not require the group itself to admit an action of type~(A), just that it embeds into a group that does.

\medskip

Let us give a new example of a family of groups satisfying the permutational Boone--Higman conjecture, and hence the Boone--Higman conjecture. Our examples come from so called shift-similar groups, introduced by Mallery and the author in \cite{mallery}. These can be viewed as an analog of self-similar groups (as in \cite{nekrashevych05}), that relate to Houghton groups rather than Thompson groups.

\begin{definition}[(Strongly) shift-similar]\label{def:shift_similar}
Consider $\Sym(\N)$, the group of permutations of $\N$. For each $j\in\N$ let $s_j \colon \N\to \N\setminus\{j\}$ be the bijection sending $i$ to $i$ for all $1\le i\le j-1$ and sending $i$ to $i+1$ for all $i\ge j$. Let $\psi_j\colon \Sym(\N)\to\Sym(\N)$ be the function defined by
\[
\psi_j(g) \coloneqq s_{g(j)}^{-1} \circ g|_{\N\setminus\{j\}} \circ s_j \text{.}
\]
Call a subgroup $G\le\Sym(\N)$ \emph{shift-similar} if for all $g\in G$ and all $j\in\N$ we have $\psi_j(g)\in G$. Call $G$ \emph{strongly shift-similar} if each $\psi_j$ restricted to $G$ is surjective.
\end{definition}

One of the most notable properties of shift-similar groups is that every finitely generated group embeds into a finitely generated, strongly shift-similar group \cite[Theorem~3.28]{mallery}. Thus, finitely generated (strongly) shift-similar groups are abundant and can be quite wild. In contrast, for finite presentability we can now prove the following:

\begin{proposition}\label{prop:shift_similar}
Every finitely presented, strongly shift-similar group admits an action of type~(A), and so in particular has solvable word problem (and satisfies the (permutational) Boone--Higman conjecture).
\end{proposition}

\begin{proof}
Let $G\le \Sym(\N)$ be finitely presented and strongly shift-similar. We claim the action of $G$ on $\N$ is of type~(A). If $G$ is finite there is nothing to prove, so assume $G$ is infinite. By \cite[Theorem~3.12]{mallery}, we know that $G$ contains $\Symfin(\N)$, the subgroup of elements of $\Sym(\N)$ that are the identity outside a finite subset. In particular, the action of $G$ on $\N$ has finitely many orbits of two-element subsets (in fact it is highly transitive). It remains only to prove that each point stabilizer is finitely generated. Let $G_j=\Stab_G(j)$. Since $G$ is strongly shift-similar, by \cite[Lemma~3.19]{mallery} the restriction of $\psi_j$ to $G_j$ is an isomorphism $\psi_j \colon G_j \to G$. Hence $G_j$ is finitely generated (even finitely presented).
\end{proof}

Actually, it is not hard to see that all stabilizers of finite subsets are isomorphic to $G$, hence finitely presented, so the results from \cite{belk22} already tell us that the corresponding twisted Brin--Thompson group is finitely presented.

As a remark, Proposition~\ref{prop:shift_similar} is somewhat analogous to the self-similar case; indeed, every finitely presented self-similar group satisfies the Boone--Higman conjecture \cite{zaremsky_fpss}. This also gives us the following, which shows that the analog of \cite[Theorem~3.28]{mallery} for finite presentability is false:

\begin{corollary}\label{cor:shift_similar}
It is not true that every finitely presented group embeds into a finitely presented strongly shift-similar group.
\end{corollary}

\begin{proof}
By Proposition~\ref{prop:shift_similar}, every finitely presented, strongly shift-similar group has solvable word problem, and so a finitely presented group with unsolvable word problem cannot embed into such a group.
\end{proof}

See \cite[Example~3.31]{mallery} for more on possible future connections between shift-similar groups and the Boone--Higman conjecture.

\medskip

We close this subsection by establishing some closure properties for groups satisfying the permutational Boone--Higman conjecture. First we prove closure under direct products.

\begin{proposition}\label{prop:direct_prod}
If two groups satisfy the permutational Boone--Higman conjecture, then so does their direct product.
\end{proposition}

\begin{proof}
A direct product of monomorphisms is a monomorphism, so it suffices to prove that if the groups $G$ and $H$ act on the sets $S$ and $T$ respectively, with both actions of type~(A), then $G\times H$ admits an action of type~(A) on some set. Indeed, the action of $G\times H$ on $S\sqcup T$, where $G$ fixes $T$ and $H$ fixes $S$, is easily checked to be of type~(A).
\end{proof}

We can also prove closure under commensurability. This proof is due to Francesco Fournier-Facio.

\begin{proposition}\label{prop:commensurable}
Let $G$ be a group. If a finite index subgroup of $G$ satisfies the permutational Boone--Higman conjecture, then so does $G$. In particular, satisfying the permutational Boone--Higman conjecture is a commensurability invariant.
\end{proposition}

\begin{proof}
Let $H \le G$ be a finite index subgroup that satisfies the permutational Boone--Higman conjecture. Say $H$ embeds into a group $E$ admitting an action of type~(A) on a set $S$. Up to passing to a deeper finite index subgroup, we may assume that $H$ is normal in $G$. By the Kaloujnine--Krasner extension theorem \cite{KK}, there is an embedding of $G$ into the regular wreath product $H \wr (G/H)$. Now consider the permutational wreath product $E \wr_n \Sigma_n$ with its natural action on $S \times \{1,\dots,n\}$. Since $\Sigma_n$ is finite and the action of $E$ on $S$ is of type~(A), it is straightforward to check that the action of $E \wr_n \Sigma_n$ on $S\times\{1,\dots,n\}$ is of type~(A). Choosing $n = |G/H|$, and embedding $G/H$ into $\Sigma_n$ via the regular action, we get embeddings
\[
G \hookrightarrow H \wr (G/H) \hookrightarrow E \wr_n \Sigma_n \text{,}
\]
proving that $G$ embeds into a group with an action of type~(A).
\end{proof}

As far as we can tell, neither of these results obviously holds for the Boone--Higman conjecture, without assuming the starting groups satisfy the permutational version. That is, it is not clear how to prove that a direct product of two finitely presented simple groups embeds into a finitely presented simple group, nor how to upgrade a virtual embedding into a finitely presented simple group to a full embedding. Some related combination-type results are unclear to us even for the permutational Boone--Higman conjecture, for example:

\begin{question}
If two groups $G$ and $H$ admit actions of type~(A) then does their free product $G*H$ admit an action of type~(A)? What about wreath products? Or general semidirect products?
\end{question}

Regarding free products, for example it is not even clear to us how to prove that the free product of two copies of Thompson's group $V$ satisfies the Boone--Higman conjecture.

\subsection{Cosets and double cosets}\label{ssec:cosets}

In this subsection, we pin down some abstract sufficient conditions for a group to admit an action of type~(A), which could potentially be useful in the future. Recall that for $H,K\le G$, a \emph{double coset} of $H$ and $K$ is a subset of $G$ of the form $HgK$ for $g\in G$, and we denote the set of all double cosets of $H$ and $K$ by $H\backslash G/K$.

\begin{proposition}\label{prop:double_cosets}
Let $G$ be a finitely presented group and $H_1,\dots,H_n\le G$ finitely generated subgroups. Suppose that the intersection of all conjugates of all the $H_i$ is trivial, and that there are finitely many double cosets in $H_i\backslash G/H_j$, for all $1\le i,j\le n$. Then $G$ admits an action of type~(A), and hence satisfies the (permutational) Boone--Higman conjecture.
\end{proposition}

\begin{proof}
Let $S=G/H_1 \sqcup \cdots \sqcup G/H_n$, with the action of $G$ by left translation. This action is faithful since the intersection of all conjugates of all the $H_i$ is trivial. It has finitely many orbits of two-element subsets since there are finitely many double cosets in $H_i\backslash G/H_j$ for each $i$ and $j$. Every point stabilizer is a conjugate of some $H_i$, hence is finitely generated. Thus, the action of $G$ on $S$ is of type~(A).
\end{proof}

Note that the $n=1$ case of Proposition~\ref{prop:double_cosets} is already interesting: we want a finitely presented group $G$ and a finitely generated subgroup $H$, such that the intersection of all conjugates of $H$ is trivial and there are finitely many double cosets in $H\backslash G/H$. In this case, compared to the $n>1$ case, it is more difficult for the intersection of conjugates to be trivial, but easier for there to be finitely many double cosets.

\medskip

One natural situation where interesting actions on sets of cosets arise is semidirect products, or more generally Zappa--Sz\'ep products. Recall that a group $G$ is the \emph{(internal) Zappa-Sz\'ep product} of subgroups $H,K\le G$ provided that $G=HK$ and $H\cap K=\{1\}$. Write $G=H\bowtie K$. This generalizes semidirect products, in that we do not require either $H$ or $K$ to be normal. Note that for all $h\in H$ and $k\in K$ there exist unique $h'\in H$ and $k'\in K$ such that $kh=h'k'$. Since $G=HK$, every double coset in $H\backslash G/H$ is of the form $HkH$ for some $k\in K$. Moreover, if $kh=h'k'$ for $h,h'\in H$ and $k,k'\in K$, then $HkH=Hk'H$. Let $\phi\colon K\times H\to K$ be the function sending $(k,h)$ to the $k'$ satisfying $kh=h'k'$ for some $h'\in H$, so this is a right action of $H$ on the set $K$.

\begin{lemma}\label{lem:zs}
Let $G=H\bowtie K$ as above, with $G$ finitely presented and $H$ finitely generated. Suppose that the action $\phi$ of $H$ on $K$ is faithful and has finitely many orbits. Then $G$ admits an action of type~(A), and hence satisfies the (permutational) Boone--Higman conjecture.
\end{lemma}

\begin{proof}
Since the action of $H$ on $K$ is faithful, for all $1\ne h\in H$ there exists $k\in K$ such that when we write $kh=h'k'$ as above, $k'\ne k$. In particular, $khk^{-1}=h'(k'k^{-1})\not\in H$. Hence, the intersection of all conjugates of $H$ is trivial, which implies that the right action of $G$ on the set of right cosets $H\backslash G$ is faithful. An arbitrary right coset looks like $Hk$ for $k\in K$. The action of $H$ on $K$ has finitely many orbits, so there are finitely many double cosets $HkH$. The result now follows from Proposition~\ref{prop:double_cosets}.
\end{proof}

We emphasize that very little needs to be assumed about $K$ here -- it does not need to be finitely generated, and the action of $H$ on $K$ does not need to be an action by group automorphisms.

In case $K$ is normal in $G$ we have $G=H\ltimes K$, and we get the following, which is worth recording.

\begin{corollary}\label{cor:semidirect}
Let $G$ be a finitely presented group. Suppose that $G$ decomposes as a semidirect product $G=H\ltimes K$ such that the action of $H$ on $K$ is faithful and has finitely many orbits. Then $G$ admits an action of type~(A), and hence satisfies the (permutational) Boone--Higman conjecture.
\end{corollary}

\begin{proof}
This is immediate from Lemma~\ref{lem:zs}, after we note that $H$ is a quotient of $G$ and hence is finitely generated.
\end{proof}

This is less clearly useful, since now the action of $H$ on $K$ is by automorphisms, and it is rather difficult for a group to have finitely many orbits under actions by automorphisms. Thus, we anticipate the general Zappa--Sz\'ep result is more likely to be useful in the future than this semidirect product result. We should also mention that, as best we can tell, the solvability of the word problem for a Zappa--Sz\'ep product (or even semidirect product) of groups with solvable word problem does not obviously follow, if the actions involved are very wild.

\subsection{Simple groups themselves}\label{ssec:simple}

A particularly intriguing question is whether satisfying the Boone--Higman conjecture is equivalent to satisfying the permuational version. In other words, we want to know whether every finitely presented simple group itself embeds into a group admitting an action of type~(A). Simple groups make it particularly easy for an intersection of conjugates of a subgroup to be trivial, so we can phrase some interesting sufficient conditions for satisfying the permutational Boone--Higman conjecture.

First recall the Boone--Higman--Thompson theorem, that every finitely generated group with solvable word problem embeds into a finitely generated simple group $H$ that embeds into a finitely presented group $G$ \cite{boone74,thompson80} (so the Boone--Higman conjecture is that this can be accomplished in a single embedding). We now have the following two sufficient conditions for satisfying the (permutational) Boone--Higman conjecture, which are immediate from Proposition~\ref{prop:double_cosets}:

\begin{corollary}
Let $H$ be a finitely generated simple subgroup of a finitely presented group $G$, and assume $H$ is not normal in $G$. If there are finitely many double cosets in $H\backslash G/H$, then $G$ admits an action of type~(A), and hence satisfies the (permutational) Boone--Higman conjecture.
\end{corollary}

\begin{proof}
Since $H$ is simple, and is not normal in $G$, the intersection of all conjugates of $H$ is trivial. The result is now immediate from Proposition~\ref{prop:double_cosets}.
\end{proof}

In other words, if the two embeddings from the Boone--Higman--Thompson theorem can be done in such a way as to ensure finitely many double cosets (and non-normality), then the conjecture holds.

\begin{corollary}
Let $G$ be a finitely presented simple group. Suppose there exists a finitely generated proper subgroup $H$ with finitely many double cosets. Then $G$ admits an action of type~(A), and hence satisfies the permutational Boone--Higman conjecture.
\end{corollary}

\begin{proof}
Since $H$ is proper and $G$ is simple, the intersection of all conjugates of $H$ in $G$ is trivial, so the result is immediate from Proposition~\ref{prop:double_cosets}.
\end{proof}

As a concluding remark, this last result applies to many existing families of finitely presented simple groups, for example it is an easy exercise to find finitely generated proper subgroups with finitely many double cosets inside Thompson's groups $T$ and $V$, and various groups related to these. It is an interesting problem, however, to try to find finitely generated proper subgroups with finitely many double cosets inside finitely presented simple groups coming from outside the extended family of Thompson-like groups. We leave this as a question, phrased more generally as follows:

\begin{question}\label{quest:BHC_PBHC}\!
\begin{enumerate}
    \item Does every finitely presented simple group admit an action of type~(A)?
    \item Does every finitely presented simple group embed in a group admitting an action of type~(A)? (And hence satisfy the permutational Boone--Higman conjecture?)
\end{enumerate}
\end{question}

We should mention that for Burger--Mozes groups, which are outside the Thompson group family, this second question was proved to be true in \cite{bux_boonehigman}, and so Burger--Mozes groups satisfy the permutational Boone--Higman conjecture. Thus, essentially the only known finitely presented simple groups outside the Thompson family for which the permutational Boone--Higman conjecture is open are the non-affine Kac--Moody groups used by Caprace and R\'emy in \cite{caprace09,caprace10}.

To be clear, if the answer to Question~\ref{quest:BHC_PBHC}(ii) is always ``yes'', then the permutational and regular Boone--Higman conjectures are equivalent, and finitely presented twisted Brin--Thompson groups are in some sense universal among finitely presented simple groups.

\bibliographystyle{alpha}

\end{document}